\documentclass[12pt,draftcls,onecolumn]{IEEEtran}                                                          

\IEEEoverridecommandlockouts                              

\usepackage{multirow}
\usepackage{wrapfig}
\usepackage{amsmath}
\usepackage[dvipsnames,table,xcdraw]{xcolor}
\usepackage{bbm}
\usepackage{diagbox}
\usepackage{varioref}
\usepackage{hyperref}
                                
\usepackage{amssymb} 
\usepackage{rotating}
\usepackage{graphicx}
\usepackage{tikz}
\usepackage{caption}
\usepackage{adjustbox}
\usepackage{subcaption}
\usepackage[normalem]{ulem}

\usepackage{amsthm}
\DeclareUnicodeCharacter{00A0}{~}
\usepackage{algorithm}
\usepackage{algpseudocode}
\algnewcommand{\Initialize}[1]{%
	\State \textbf{Initialization:}
	\Statex {\raggedright #1}
}
\newtheorem{assumption}{Assumption}
\newtheorem{theorem}{Theorem}
\newtheorem{lemma}{Lemma}
\newtheorem{proposition}{Proposition}
\newtheorem{definition}{Definition}

\newtheorem{remark}{Remark}

\definecolor{ao}{rgb}{0.0, 0.5, 0.0}

\newcommand{\fy}[1]{{\color{black}#1}}
\newcommand{\yqr}[1]{{\color{black}#1}}

\usepackage[textsize=small]{todonotes}

\newcommand{\yq}[1]{{\color{black}#1}}

\title{\LARGE \bf
A \fy{Randomized Zeroth-Order} Hierarchical Framework for Heterogeneous Federated Learning
}

\author{Yuyang Qiu, Kibaek Kim, and Farzad Yousefian 
\thanks{This work  is supported in part by the U.S. Department of Energy, Office of Science, Advanced Scientific Computing Research, under Grant DE-SC0025570 and Contract DE-AC02-06CH11357, and in part by the Office of Naval Research under Grant N00014-22-1-2757.}
\thanks{Yuyang Qiu and Farzad Yousefian are with the Department of Industrial and \yqr{S}ystems Engineering, Rutgers University, New Brunswick, NJ 08901, USA. {\tt\small \{yuyang.qiu, farzad.yousefian\}@rutgers.edu}}
\thanks{Kibaek Kim is with the Mathematics and Computer Science Division, Argonne National Laboratory, Lemont, IL 60439, USA. {\tt\small kimk@anl.gov}}%
}

\begin{document}
\maketitle
\thispagestyle{empty}
\pagestyle{empty}
\sloppy

\begin{abstract}
Heterogeneity in \fy{f}ederated \fy{l}earning (FL) is a critical and challenging aspect that significantly impacts model performance and convergence. In this paper, we propose a novel framework by formulating heterogeneous FL as \fy{a} hierarchical optimization problem. This new framework captures both local and global training process\yqr{es} through a bilevel formulation and is capable of \fy{the following}: (i) addressing client heterogeneity \fy{through a personalized learning framework}; (ii) capturing \yqr{the} pre-training process on \yqr{the server} side; (iii) updating \yqr{the} global model through nonstandard aggregation; (iv) allowing for nonidentical local steps\fy{; and (v)} capturing clients' local constraints. We design and analyze an implicit zeroth-order FL method (ZO-HFL), \yqr{equipped} with nonasymptotic convergence guarantees for both the server-agent and the individual client-agents, and asymptotic guarantees for both the server-agent and client-agents in an almost sure sense. Notably, our \fy{method} does not \fy{rely on} standard assumptions in heterogeneous FL\fy{,} such as \fy{the} bounded \fy{gradient} dissimilarity condition. We implement our method on image classification tasks and compare with other methods under \fy{different} heterogeneous settings.
\end{abstract}


\section{Introduction}
Federated learning (FL) \cite{mcmahan2017communication}, as a decentralized, communication-efficient learning paradigm, enables participating clients to obtain a generalizable model \yqr{without sharing raw data}. One of the primary challenges in FL is the presence of client heterogeneity. Effectively addressing heterogeneity is crucial for ensuring robust performance, fairness, and generalization of the global model across all participating clients. Heterogeneity impacts FL in several critical ways as follows: \fy{(i)} \textit{Data heterogeneity:} When \yqr{clients' data is not necessarily independent and identically distributed, commonly referred to as non-iid}, traditional aggregation methods may lead to biased global models that underperform on certain client data \cite{li2020federated}. \fy{(ii)} \textit{System heterogeneity}: Devices involved in FL often have varying computational power and network conditions, making synchronous aggregation challenging and potentially leading to straggler effects \cite{bonawitz2019towards}. \fy{(iii)} \textit{Model heterogeneity}: Different clients may require tailored models, especially in personalized FL scenarios \cite{smith2017federated}.

Various strategies have been proposed to tackle heterogeneity in FL. Regularization schemes \cite{li2020federated} and variance control techniques \cite{karimireddy2020scaffold}, have been employed to handle non-iid data and system inconsistencies. Some FL methods are equipped with asynchronous updates \cite{FedBuff} or nonidentical local steps \cite{wang2020tackling}, \fy{allowing clients} to update and communicate based on local computation and network constraints. Personalized approaches \cite{tan2022towards} such as meta-learning \cite{fallah2020personalized}, clustering \cite{sattler2020clustered}, and model-remapping \cite{lu2024zoopfl} allow adaptation to client-specific data distributions. We note that some existing personalized FL frameworks address heterogeneity through bilevel optimization, such as fairness and robustness approach \cite{fairpersonalization}, sparse personalized approach \cite{liu2023sparse}, and adaptive mixed model approach \cite{deng2020adaptive}.
Additionally, some nonstandard aggregation steps \yqr{on the server side} were proposed, such as server-side momentum \cite{AgnosticFL19}, extrapolation mechanism \cite{jhunjhunwala2023fedexp}, and robust aggregation \cite{pillutla2022robust}. However, \yqr{since the server could also preserve data, and} there seems to be no FL framework that considers training with both \yqr{client and server} data. Therefore, an open problem arises:
{\it Can we design a framework that captures training on clients and server while addressing clients' heterogeneity?} To this end, we propose a novel \fy{modeling} framework to address heterogeneity in FL, while capturing server-side pre-training.

\yqr{Since we consider utilizing data on the server side, we associate the server with a loss function $f_1(x)= \mathbb E [\tilde f_1(x,\xi)]$, where $x$ and $\xi\in \mathcal D $ denote the global model and data samples maintained at the server, respectively. Next, let $h_i(x,y_i)= \mathbb E [\tilde h_i(x,y_i,\zeta_i)]$, for all $i\in[m]$, denote the $i$-th client's local objective function, where $y_i$ and $\zeta_i\in \tilde{\mathcal D}_i $ denote the local model and data samples maintained by the client, respectively. We let $[m]$ denote the set of integers from $1$ to $m$. Then, we define a penalty function $f_2(x,y_i(x))$ that is utilized to penalize the dissimilarity between server and client models, where $y_i(x)= \text{arg} \min_{y_i\in Y_i(x)}  \mathbb E [\tilde h_i(x,y_i,\zeta_{i})]$. For example, one may consider $f_2(x,y_i(x))=\tfrac{\lambda}{2}\|x-y_i(x)\|^2$,  where $\lambda>0$ is a penalty parameter.
Then, we may} consider a distributed hierarchical optimization problem of the form
\begin{align}\label{prob::main}
\begin{split}
&\min_{x\in \mathbb R^n}  \quad \mathbb E [\tilde f_1(x,\xi)]+\tfrac{1}{m}\textstyle\sum_{i=1}^m  f_2(x,y_i(x))\\
& \text{ s.t. }\quad y_i(x) = \text{arg} \min_{y_i\in Y_i(x)}  \mathbb E [\tilde h_i(x,y_i,\zeta_{i})], \  \forall i\in[m],
\end{split}
\end{align}
where the upper level is associated with a server-agent, and the lower level is associated with $m$ client-agents. We define $ f_2(\bullet) \triangleq \frac{1}{m}\sum_{i=1}^m  f_2(\bullet,y_i(\bullet))$ to denote the implicit function, \fy{mitigating} the drift of local clients. $Y_i(x) \subseteq \mathbb R^n$ for all $i\in [m]$ denotes the local constraint sets in the \fy{lower level}.
\fy{We consider a flexible setting} where each client-agent \fy{is locally constrained by} $Y_i(x)$ \fy{enabling} personalization, \fy{e.g., when} $Y_i(x)=\{y_i \fy{\in \mathbb{R}^n \mid} \|y_i-x\|\leq \rho_i\}$ \fy{where $\rho_i$ denotes the $i$th client's local dissimilarity bound}. Our proposed framework not only captures a global model $x$ at server's end, but also allows each client \yqr{$i$} to maintain \yqr{a} personalized local model $y_i$ characterized by \yqr{the} local \yqr{dataset} $\tilde{\mathcal D} _i$, \yqr{making} it particularly suitable for heterogeneous FL scenarios. Throughout, we denote the server's global objective by $f(x)=f_1(x)+f_2(x)$. We note that \fy{problem} \eqref{prob::main} is distinct from standard bilevel FL formulations, such as \fy{those studied in}~\cite{yang2024simfbo,qiu2023zeroth} \fy{where the hierarchical structure arises from the learning task (e.g., hyperparameter optimization), but not the heterogeneous  setting}.

\fy{Our main contributions are summarized as follows.}
(i) We design and analyze a randomized zeroth-order implicit heterogeneous FL method, {\bf ZO-HFL}, for addressing problem \eqref{prob::main}. Notably, our method utilizes nonstandard aggregations and allows nonidentical local steps.  (ii) We provide nonasymptotic guarantees for the setting when the implicit function is nondifferentiable \yq{and} nonconvex. (iii) We also provide almost sure convergence with asymptotic guarantees for both upper-level and lower-level objectives. (iv) We numerically validate our theoretical findings under various settings, and achieve better test accuracy under highly heterogeneous settings.

\noindent {\bf Notation.} Throughout, we let $\|\bullet\|$ \fy{denote} the $\ell_2$ norm, and $x^{\top}y$ \fy{denote} the inner product \yqr{of vectors} $x, y \in \mathbb R^n$. We use $\Pi_C[x]$ to denote the Euclidean projection of point $x$ onto \yqr{the nonempty, closed, and convex} set $C$. \yqr{The nonexpansivity property of the Euclidean projection is as follows.} $\|\Pi_C[x]-\Pi_C[y]\|\leq \|x-y\|$ for all $x,y\in \mathbb R^n$. We let
$\mathbb{B}$ and $\mathbb{S}$ denote the $n$-dimensional unit ball and its
surface, respectively, i.e., $\mathbb{B} = \{u \in \mathbb{R}^n \mid \|u\| \leq
1\}$ and $\mathbb{S} = \{v \in \mathbb{R}^n \mid \|v\| = 1\}$. Throughout, we use $\mathbb{E}[\bullet]$ to denote the expectation of a random variable. We denote the big $O$ notation by $\mathcal{O}\left(\bullet\right)$.
We say a \yqr{continuously differentiable} function $f: C\subseteq \mathbb R^n\to \mathbb R$ defined on a convex set $C$ is $\mu$-strongly convex if \yqr{$f(y)\geq f(x)+\nabla f(x)^{\top}(y-x)+\tfrac{\mu}{2}\|y-x\|^2$} for all $x,y\in C$; and is $L$-smooth if $f(y)\leq f(x)+\nabla f(x)^{\top}(y-x)+\tfrac{L}{2}\|y-x\|^2$, for all $x,y\in C$. 

\section{Assumption and Algorithm Outline}
In this section, we present the main assumptions and outline of the proposed FL algorithm. 


\begin{assumption}\label{assump:main}\em
Let the following assumptions hold.
 
\noindent (i) For all $i \in [m]$, $f_2(\bullet,y_i(\bullet))$ is $L_0^{\text{imp}}$-Lipschitz, $f_2(\bullet,y)$ is $L_{0,x}^{f_2}$-Lipschitz \yqr{for} any $y$, and $f_2(x,\bullet)$ is $L_{0,y}^{f_2}$-Lipschitz 
\yqr{for} any $x$. 

\noindent (ii) For all $i\in [m]$, for any $x$, $h_i(x,\bullet)\triangleq \mathbb{E}_{\zeta_i \in\tilde{\mathcal{D}}_i}[\tilde{h}_i(x,\bullet,\zeta_i) ]$ is $L_{1,y}^{h_i}$-smooth and $\mu_{h_i}$-strongly convex. \fy{For} any $y$, the map $\nabla_y h_i(\bullet,y)$ is Lipschitz continuous with parameter $L_{0,x}^{\nabla h_i}$. Furthermore, $\mathbb E[\nabla_y \tilde h_i(x,y_{i},\zeta_i)\mid y_{i}]=\nabla_y h_i(x,y_i)$ and $\mathbb E[\|\nabla_y h_i(x,y_{i})-\nabla_y \tilde h_i\fy{(x,y_{i},\zeta_i)}\|^2\mid y_{i}]\leq \sigma_i^2$.

\noindent (iii) For any $x$, $f_1(x)$ is $L^{f_1}_{1}$-smooth, $\mathbb E[\nabla \tilde f_1(x,\xi)\mid x]=\nabla f_1(x)$, and $\mathbb E[\|\nabla \tilde f_1(x,\xi)-\nabla f_1(x)\|^2\mid x]\leq \sigma^2$.

\noindent (iv) \yqr{For all $i \in [m]$, for any $x$,} $Y_i(x)$ \yqr{is} closed and convex.
\end{assumption}

\begin{remark}
Note that Assumption \ref{assump:main} does not require standard assumptions made in heterogeneous FL, which is critical in establishing convergence, and is criticized in capturing data heterogeneity \cite{wang2022unreasonable}.
\end{remark}

\yq{The outline of our methods is described as follows. We employ \yqr{the} implicit programming approach \cite{qiu2023zeroth,cui2023complexity} for solving problem \eqref{prob::main}. During local steps, client-agents update their local models by solving the lower-level problem in \eqref{prob::main} with local solver (e.g., projected SGD). Then, after receiving information from clients, the server-agent updates the global model \yqr{for minimizing the global implicit objective} $f(x)=f_1(x)+f_2(x)$.}

A major challenge in hierarchical optimization is that the {\it implicit function is often nondifferentiable and nonconvex}, especially when the lower-level problems are constrained \cite{qiu2023zeroth}. This is shown using an example in the appendix (Fig. \ref{fig:NDNCBiLv}). Therefore, we employ a {\it randomized smoothing} scheme on the implicit function $f_2(\bullet)$. \yqr{Randomized smoothing is rooted in the seminal work
by Steklov~\cite{steklov1}, which has inspired significant progress in both
convex~\cite{DeFarias08,YNS_Automatica12,Duchi12} and nonconvex~\cite{lin2022gradient,nesterov17}
(single-level) optimization. It was later extended to address nonsmooth, nonconvex, and hierarchical problems \cite{qiu2023zeroth,cui2023complexity}.}

\yqr{Given a smoothing parameter $\eta>0$, we define (cf. Lemma~\ref{lem:SphericalSmooth}) the smoothed variant} of $f_2(x)$ as
\begin{align*}
f_2^{\eta}(x)\triangleq \tfrac{1}{m}\textstyle\sum_{i=1}^m\mathbb E_{u\in\mathbb B}[f_2(x+\eta u,y_i(x+\eta u))].
\end{align*}
We also define $f^{\eta}(x)=f_1(x)+f_2^{\eta}(x)$ and $\nabla f_2^{\eta}(x)=\tfrac{1}{m}\sum_{i=1}^m \mathbb E_{v\in\mathbb S}[(f_2(x+\eta v,y_i(x+\eta v))-f_2(x-\eta v,y_i(x-\eta v)))v]$. \yqr{Because} the exact evaluation of $y_i(\bullet)$ is typically intractable \cite{qiu2023zeroth,cui2023complexity}, we use an inexact evaluation $y_{\varepsilon_i}(\bullet)$ \yqr{such that} $\mathbb E[\|y_i(\bullet)-y_{\varepsilon_i}(\bullet)\|^2] \leq \varepsilon_i$. Therefore, we may utilize a stochastic inexact zeroth-order gradient of $f^\eta(x)$, given as
\begin{align*}
\hat g_\varepsilon &=     \nabla \tilde f_1(x,\xi)+\tfrac{1}{m} \textstyle\sum_{i=1}^m \tfrac{n}{2\eta}(f_2(x+\eta v_{i},y_{\varepsilon_{i}}(x+\eta v_i))-f_2(x-\eta v_{i},y_{\varepsilon_{i}}(x-\eta v_i)))v_{i}\yqr{,}
\end{align*}
where $v_{i}\in \eta \mathbb S$.
\yq{Then, we} consider a gradient-based global step at round $r$ given as $\hat x_{r+1}=\hat x_r - \gamma_r \hat g_{\varepsilon,r}$, where $\hat g_{\varepsilon,r}$ denotes the realization of $\hat g_\varepsilon$ at round $r$. The \yqr{outline} of the proposed FL scheme \yqr{is} presented in Algorithms~\ref{alg:FedStack} and~\ref{alg:SGD}.
\begin{algorithm}[htb]                                                                       \caption{{\bf ZO-HFL}}\label{alg:FedStack}
\begin{algorithmic}[1] 
	 \State \textbf{Initialization:} server obtains initial global model $\hat x_0$
	 \For {$r = 0,1, \dots, R-1$}
	 \State server generates $v_{i,r}\in \mathbb S$, and broadcast\yqr{s} $\hat x_r$ and $v_{i,r}$ to clients
	 
	 \For {$i = 1,  \dots,m$}
	 \State client $i$ calls Algorithm \ref{alg:SGD} twice, and obtains $y_{\varepsilon_{i,r}}^+\triangleq y_{\varepsilon_{i,r}}(\hat x_r+\eta v_{i,r})$ and $y_{\varepsilon_{i,r}}^-\triangleq y_{\varepsilon_{i,r}}(\hat x_r-\eta v_{i,r})$, and sends them to server
	 \EndFor
	 \State server generates a random sample $\xi_{r}\in \mathcal D$ and computes $\hat g_{\varepsilon,r}=\nabla \tilde f_1(\hat x_r,\xi_r)+\tfrac{1}{m} \sum_{i=1}^m g_{i,r}^{\eta,\varepsilon}$, where $g_{i,r}^{\eta,\varepsilon}=\tfrac{n}{2\eta}(f_2(\hat x_r+\eta v_{i,r},y_{\varepsilon_{i,r}}^+)-f_2(\hat x_r-\eta v_{i,r},y_{\varepsilon_{i,r}}^-))v_{i,r}$
	 \State server updates \yqr{the global model as} $\hat x_{r+1} = \hat x_r -\gamma_r \hat g_{\varepsilon,r}$
	 \EndFor
	\State return $\hat x_{R}$
	\end{algorithmic}
\end{algorithm}

\begin{algorithm}[htb]                                                                       \caption{Client $i$'s local steps ($i$, $r$, $x$, $H_{i,r}$, $\tilde \gamma_{i,t}$)}\label{alg:SGD}
\begin{algorithmic}[1]
	 \State \textbf{Initialization:} client $i$ choose\yqr{s} initial point $y_{i,0}^{r,\bullet}$
	 \For {$t = 0,1, \dots, H_{i,r}-1$}
	 \State client $i$ generates a random sample $\zeta_{i,t}^{r,\bullet}\in \mathcal{ \tilde D}_i$
	 \State $y_{i,t+1}^{r,\bullet}= \Pi_{Y_i(x)}[y_{i,t}^{r,\bullet}-\tilde \gamma_{i,t} \nabla_y \tilde h_i(x,y_{i,t}^{r,\bullet},\zeta_{i,t}^{r,\bullet})]$
	 \EndFor
	\State return $y_{i,H_{i,r}}^{r,\bullet}$ as $y_{\varepsilon_{i,r}}^\bullet$
	\end{algorithmic}
\end{algorithm}

\begin{remark}
In Alg.~\ref{alg:FedStack}, \yqr{$\hat x_r$ denotes the global model at round $r$, $v_{i,r}$ denotes the random vector generated uniformly on $\mathbb S$,} $y_{\varepsilon_{i,r}}(\bullet)\triangleq y_{i,H_{i,r}}^{r,\bullet}$ denotes an $\varepsilon_{i,r}$-solution \yqr{such that} $\mathbb E[\|y_{i,H_{i,r}}^{r,\bullet}-y_{i,*}^{r,\bullet}\|^2]\leq \varepsilon_{i,r}$, where $y_{i,*}^{r,\bullet}$ denotes the optimal solution of the lower-level problem $\min_{y_i\in Y_i(x)}  \mathbb E [\tilde h_i(\bullet,y_i,\zeta_{i})]$. We let $y_{i,t}^{r,\bullet}$  denote client $i$'s local iterates at round $r$, given $\hat x_r+\eta v_{i,r}$ or $\hat x_r-\eta v_{i,r}$ as the input variable $x$ in Alg.~\ref{alg:SGD} (e.g., if the iterate is $y_{i,t}^{r,+}$, it means Alg. \ref{alg:SGD} receives $x_r+\eta v_{i,r}$ as the input variable $x$). \yqr{We note that in Alg. 2, setting $H_{i,r}=0$ for some clients at some rounds may simulate system heterogeneity effect.}
\end{remark}

\section{Convergence Analysis}
In this section, we analyze the convergence of the proposed scheme. We begin by defining the method's history. 
\begin{definition}\em\label{def:F:Alg}
We first define the history of Algorithm \ref{alg:SGD} at round $r$, for all $i\in [m]$ and $1 \leq t \leq H_{i,r}$. $\mathcal F_{i,t}^{r,\bullet} \triangleq \mathcal F_{i,t-1}^{r,\bullet} \cup \{\zeta_{i,t-1}^{r,\bullet}\}.$
Let $\mathcal F_{i,0}^{r,\bullet}\triangleq \{y_{i,0}^{r,\bullet}\}\cup (\cup_{i=1}^m\cup_{j=0}^r \{v_{i,j}\})\cup \mathcal F_{r}$, \yqr{where $\mathcal{F}_r$ denotes} the history of Algorithm \ref{alg:FedStack} \yqr{ defined next.} \yqr{F}or all $r\geq 1$, $
\mathcal F_r\triangleq (\cup_{i=1}^m\{\mathcal F_{i,H_{i,r-1}}^{r-1,+},\mathcal F_{i,H_{i,r-1}}^{r-1,-}\})\cup \{\xi_{r-1}\},$
where $\mathcal F_{i,H_{i,r}}^{r,+}$ and $\mathcal F_{i,H_{i,r}}^{r,-}$ are defined above. \yqr{Lastly, we} let $\mathcal F_0 \triangleq \{\hat x_0\}\cup (\cup_{i=1}^m\{\mathcal F_{i,H_{i,0}}^{0,+},\mathcal F_{i,H_{i,0}}^{0,-}\})$.
\end{definition}

Next, we \yqr{review} an important result in establishing almost sure convergence of stochastic methods.
\begin{lemma}[Robbins-Siegmund Theorem \cite{robbins1971convergence}]\em\label{lem:robsieg}
For $t=0,1,\dots$, let $X_t$, $Y_t$, $Z_t$, and $\alpha_t$ be nonnegative $\tilde{\mathcal F}_t$-measurable random variables, where $\tilde{\mathcal F}_t\subset \tilde{\mathcal F}_{t+1}$. Suppose the following relations hold: (a) $\mathbb E[Y_{t+1}\mid \tilde{\mathcal F}_t]\leq (1+\alpha_t)Y_t -X_t +Z_t$; (b) $\sum_{t=0}^{\infty} Z_t < \infty$, $\sum_{t=0}^{\infty} \alpha_t < \infty$. Then we have
$\lim_{t\to \infty} Y_t = Y\geq 0, \text{ and }\sum_{t=0}^{\infty} X_t<\infty \text{ \yqr{almost surely.}}$

\end{lemma}
Next, we derive a nonasymptotic error bound for Algorithm~\ref{alg:SGD} that characterizes the inexactness.
\begin{proposition}\label{prop:SGD}\em
Consider Algorithm \ref{alg:SGD}\fy{. Let} Assumption~\ref{assump:main} (ii) hold. Let $y_{i,*}^{r,\bullet}$ denote the optimal solution \fy{to} the lower-level problem $\min_{y_i\in Y_i(\bullet)}  \mathbb E [\tilde h_i(\bullet,y_i,\zeta_{i})]$,  \fy{and} $g_{i,t}^{r,\bullet}\triangleq \nabla_y \tilde h_i(\bullet,y_{i,t}^{r,\bullet},\zeta_{i,t}^{r,\bullet})$\fy{. Then,} the following results hold \yqr{for all $i \in [m]$}.

\noindent (i) Let $\tilde\gamma_{i,t}\triangleq\tfrac{\tilde\gamma_i}{t+\Gamma}\leq \tfrac{\mu_{h_i}}{2(L_{1,y}^{h_i})^2}$, where $\Gamma, \tilde\gamma_i>0$. Then
\begin{align*} 
&\mathbb E[\|y_{i,H_{i,r}}^{r,\bullet}-y_{i,*}^{r,\bullet}\|^2]\leq \tfrac{\max \{\tfrac{2\sigma_i^2\tilde\gamma_i^2}{\mu_{h_i} \tilde\gamma_i-1},\Gamma \|y_{i,0}^{r,\bullet}-y_{i,*}^{r,\bullet}\|^2 \}}{H_{i,r}+\Gamma}.
\end{align*}
\noindent (ii) Let $\{\tilde\gamma_{i,t}\}$ be a nonsummable and square summable sequence. We have
$\lim_{t \to \infty} \|y_{i,t}^{r,\bullet}-y_{i,*}^{r,\bullet}\| = 0,$ \yqr{ almost surely.}

\end{proposition}
\begin{proof}
(i) By the \yqr{nonexpansiveness} property of the projection mapping, we have
\begin{align*}
&\|y_{i,t+1}^{r,\bullet}-y_{i,*}^{r,\bullet}\|^2 
= \|\Pi_{Y_i(x)}[y_{i,t}^{r,\bullet}-\tilde \gamma_{i,t} g_{i,t}^{r,\bullet}]-\Pi_{Y_i(x)}[y_{i,*}^{r,\bullet}-\tilde\gamma_{i,t}\nabla_y h_i(x,y_{i,*}^{r,\bullet})]\|^2 \\
&\leq \|y_{i,t}^{r,\bullet}-\tilde \gamma_{i,t} g_{i,t}^{r,\bullet}-(y_{i,*}^{r,\bullet}-\tilde\gamma_{i,t}\nabla_y h_i(x,y_{i,*}^{r,\bullet}))\|^2 \\
&=\|y_{i,t}^{r,\bullet}-y_{i,*}^{r,\bullet}\|^2+\tilde \gamma_{i,t}^2\|\nabla_y h_i(x,y_{i,*}^{r,\bullet})-g_{i,t}^{r,\bullet}\|^2 + 2\tilde\gamma_{i,t}(y_{i,t}^{r,\bullet}-y_{i,*}^{r,\bullet})^{\top} (\nabla_y h_i(x,y_{i,*}^{r,\bullet})-g_{i,t}^{r,\bullet}).
\end{align*}
Taking conditional expectation\yqr{s} on both sides, we obtain
\begin{align*}
&\mathbb E[\|y_{i,t+1}^{r,\bullet}-y_{i,*}^{r,\bullet}\|^2 \mid \mathcal F_{i,t}^{r,\bullet}]\\
&\leq \|y_{i,t}^{r,\bullet}-y_{i,*}^{r,\bullet}\|^2+ \tilde \gamma_{i,t}^2\mathbb E[\|\nabla_y h_i(x,y_{i,*}^{r,\bullet})-g_{i,t}^{r,\bullet}\|^2\mid \mathcal F_{i,t}^{r,\bullet}]\\
&-2\tilde\gamma_{i,t}(y_{i,t}^{r,\bullet}-y_{i,*}^{r,\bullet})^{\top} \mathbb E[g_{i,t}^{r,\bullet}-\nabla_y h_i(x,y_{i,*}^{r,\bullet})\mid \mathcal F_{i,t}^{r,\bullet}].
\end{align*}
Next, add and subtract $\nabla_y h_i(x,y_{i,t}^{r,\bullet})$ in the second term of the previous relation, and utilize $\mathbb E[g_{i,t}^{r,\bullet}\mid \mathcal F_{i,t}^{r,\bullet}]=\nabla_y h_i(x,y_{i,t}^{r,\bullet})$\fy{. We have}
\begin{align*}
&\mathbb E[\|y_{i,t+1}^{r,\bullet}-y_{i,*}^{r,\bullet}\|^2 \mid \mathcal F_{i,t}^{r,\bullet}]\leq \|y_{i,t}^{r,\bullet}-y_{i,*}^{r,\bullet}\|^2+ 2\tilde\gamma_{i,t}^2\|\nabla_y h_i(x,y_{i,*}^{r,\bullet})-\nabla_y h_i(x,y_{i,t}^{r,\bullet})\|^2\\
&+2\tilde\gamma_{i,t}^2\mathbb E[\|\nabla_y h_i(x,y_{i,t}^{r,\bullet})-g_{i,t}^{r,\bullet}\|^2\mid \mathcal F_{i,t}^{r,\bullet}]-2\tilde\gamma_{i,t}(y_{i,t}^{r,\bullet}-y_{i,*}^{r,\bullet})^{\top}(\nabla_y h_i(x,y_{i,t}^{r,\bullet})-\nabla_y h_i(x,y_{i,*}^{r,\bullet})).
\end{align*}
Then, \yqr{we} utilize Assumption \ref{assump:main} (ii) \yqr{implying} that $h_i(x,\bullet)$ is $\mu_{h_i}$-strongly convex and $L_{1,y}^{h_i}$-smooth,  and $\mathbb E[\|\nabla_y h_i(x,y_{i,t}^{r,\bullet})-g_{i,t}^{r,\bullet}\|^2\mid \mathcal F_{i,t}^{r,\bullet}]\leq \sigma_i^2$. \yqr{W}e obtain
\begin{align*}
&\mathbb E[\|y_{i,t+1}^{r,\bullet}-y_{i,*}^{r,\bullet}\|^2 \mid \mathcal F_{i,t}^{r,\bullet}]\\
&\leq \|y_{i,t}^{r,\bullet}-y_{i,*}^{r,\bullet}\|^2+ 2\tilde\gamma_{i,t}^2(L_{1,y}^{h_i})^2\|y_{i,t}^{r,\bullet}-y_{i,*}^{r,\bullet}\|^2+2\tilde\gamma_{i,t}^2\sigma_i^2-2\tilde\gamma_{i,t}\mu_{h_i}\|y_{i,t}^{r,\bullet}-y_{i,*}^{r,\bullet}\|^2\\
&=(1+2\tilde\gamma_{i,t}^2(L_{1,y}^{h_i})^2-2\tilde\gamma_{i,t}\mu_{h_i})\|y_{i,t}^{r,\bullet}-y_{i,*}^{r,\bullet}\|^2+2\tilde\gamma_{i,t}^2\sigma_i^2.
\end{align*}
Let $0<\tilde\gamma_{i,t}\leq \tfrac{\mu_{h_i}}{2(L_{1,y}^{h_i})^2}$\fy{. We} have $1+2\tilde\gamma_{i,t}^2(L_{1,y}^{h_i})^2-2\tilde\gamma_{i,t}\mu_{h_i}\leq 1-\tilde\gamma_{i,t}\mu_{h_i}<1$. \fy{We obtain}
\begin{align}\label{eqn:sgd:recur}
&\mathbb E[\|y_{i,t+1}^{r,\bullet}-y_{i,*}^{r,\bullet}\|^2 \mid \mathcal F_{i,t}^{r,\bullet}]\nonumber\leq (1-\tilde\gamma_{i,t}\mu_{h_i})\|y_{i,t}^{r,\bullet}-y_{i,*}^{r,\bullet}\|^2+2\tilde\gamma_{i,t}^2\sigma_i^2.
\end{align}
Next, \yqr{substituting} $\tilde\gamma_{i,t}$ \yqr{by} $\tfrac{\tilde\gamma_i}{t+\Gamma}$, where $\Gamma, \tilde\gamma_i>0$, \yqr{and invoking} \cite[Lemma 8]{cui2023complexity}, we obtain \fy{for all} $ k\geq 0$
\begin{align*} 
&\mathbb E[\|y_{i,t}^{r,\bullet}-y_{i,*}^{r,\bullet}\|^2 \mid \mathcal F_{i,t}^{r,\bullet}]\leq \tfrac{\max \left\{\frac{2\sigma_i^2\tilde\gamma_i^2}{\mu_{h_i} \tilde\gamma_i-1},\Gamma \|y_{i,0}^{r,\bullet}-y_{i,*}^{r,\bullet}\|^2 \right\}}{t+\Gamma}.
\end{align*}
Taking total expectation\yqr{s} on both sides and let\yqr{ting} $t:=H_{i,r}$, we obtain the desired result. \yq{We note that $\|y_{i,0}^{r,\bullet}-y_{i,*}^{r,\bullet}\|^2$ remains bounded if $\forall i\in [m]$ and $r\geq 0$: (i) $y_{i,0}^{r,\bullet}\in \mathcal B_{y,0}\cap Y_i(x)$, where $\mathcal B_{y,0}$ is a compact set; (ii) $\forall x\in \mathbb R^n$, $\cup_{i=1}^m \yqr{Y}_i(x)$ is bounded.}

\noindent (ii) By invoking equation \eqref{eqn:sgd:recur} and \fy{letting} $X_t=\tilde\gamma_{i,t}\mu_{h_i}\|y_{i,t}^{r,\bullet}-y_{i,*}^{r,\bullet}\|^2$, $Y_t=\|y_{i,t}^{r,\bullet}-y_{i,*}^{r,\bullet}\|^2$, and $Z_t=2\tilde\gamma_{i,t}^2\sigma_i^2$, we \yq{can} apply Lemma \ref{lem:robsieg} and obtain $\sum_{t=0}^{\infty} \tilde\gamma_{i,t}\|y_{i,t}^{r,\bullet}-y_{i,*}^{r,\bullet}\|^2<\infty$ almost surely. \yq{The rest of the \fy{proof} can be done in a similar fashion as in \cite[Prop. 1]{YNS_Automatica12}.}
\end{proof}


\begin{lemma}[L\'evy concentration on sphere $\mathbb S$ {\cite{wainwright2019high}}]\label{lem:Levy}\em
Let $h:\mathbb{R}^n\to \mathbb{R}$ be a given $L_0$-Lipschitz continuous function, and $v$ \yqr{be} uniformly distributed on sphere $\mathbb S$. Then, we have
$$
\mathbb P(|h(v)-\mathbb E[h(v)]|\geq \epsilon) \leq 2\sqrt{2\pi} e^{-\tfrac{n\epsilon^2}{8L_0^2}},\, \forall \epsilon>0. 
$$
\end{lemma}

\begin{lemma}\label{lem:lip:hv}\em
Let Assumption \ref{assump:main} (i) hold. Define $h_i(v)=f_2(x+\eta v,y_i(x+\eta v))$, where $\eta>0$, $v\in\mathbb S$, and $x\in \mathbb R^n$. Then\fy{, for all $i \in [m]$}, $h_i$ is $\eta L_0^{\text{imp}}$-Lipschitz continuous.
\end{lemma}
\begin{proof} \fy{We have}
\begin{align*}
&|h_i(v_1)-h_i(v_2)|\\
&= |f_2(x+\eta v_1,y_i(x+\eta v_1)) - f_2(x+\eta v_2,y_i(x+\eta v_2))|\\
& \leq |f_2(x+\eta v_1,y_i(x+\eta v_1)) - f_2(x+\eta v_1,y_i(x+\eta v_2))|\\
& + |f_2(x+\eta v_1,y_i(x+\eta v_2)) - f_2(x+\eta v_2,y_i(x+\eta v_2))|\\
&\stackrel{\text{Assumption }\ref{assump:main}\text{ (i)}}{\leq} L_{0,y}^{f_2}\|y_i(x+\eta v_1)-y_i(x+\eta v_2)\|+ L_{0,x}^{f_2}\|\eta v_1 -\eta v_2\|.
\end{align*}
We can show that $y_i(\bullet)$ is $\tfrac{L_{0,x}^{\nabla h_i}}{\mu_{h_i}}$-Lipschitz continuous and $L_0^{\text{imp}}=\tfrac{L_{0,y}^{f_2} L_{0,x}^{\nabla h_i}}{\mu_{h_i}}+ L_{0,x}^{f_2}$ \yq{in a similar fashion as in} \cite[Lemma 2]{qiu2023zeroth}. Then, we obtain
\begin{align*}
|h_i(v_1)-h_i(v_2)|&\leq \eta(\tfrac{L_{0,y}^{f_2} L_{0,x}^{\nabla h_i}}{\mu_{h_i}}+ L_{0,x}^{f_2})\|v_1 - v_2\|\\
&=\eta L_0^{\text{imp}}\|v_1 - v_2\|.
\end{align*}
Therefore, $h_i(v)$ is $\eta L_0^{\text{imp}}$-Lipschitz continuous.
\end{proof}

Next, we present some definitions for the analysis.
\begin{definition}\em\label{def:main}
Let us define the following terms.
\begin{align*}
& f_{2,i,r}^{+}= f_2(\hat x_r+\eta v_{i,r},y_i(\hat x_r+\eta v_{i,r})), \\
& f_{2,i,r}^{-}= f_2(\hat x_r-\eta v_{i,r},y_i(\hat x_r-\eta v_{i,r})), \\
& f_{2,i,r}^{\varepsilon,+}= f_2(\hat x_r+\eta v_{i,r},y_{\varepsilon_{i,r}}(\hat x_r+\eta v_{i,r})), \\
& f_{2,i,r}^{\varepsilon,-}= f_2(\hat x_r-\eta v_{i,r},y_{\varepsilon_{i,r}}(\hat x_r-\eta v_{i,r})), \\
& g_{i,r}^{\eta}=\tfrac{n}{2\eta}(f_{2,i,r}^{+} - f_{2,i,r}^{-})v_{i,r}, \,\,\, g_{i,r}^{\eta,\varepsilon}=\tfrac{n}{2\eta}(f_{2,i,r}^{\varepsilon,+} - f_{2,i,r}^{\varepsilon,-})v_{i,r},\\
& \bar g_{r}^{\eta}=\tfrac{1}{m}\textstyle\sum_{i=1}^m g_{i,r}^{\eta}, \quad \bar g_{r}^{\eta,\varepsilon}=\tfrac{1}{m}\textstyle\sum_{i=1}^m g_{i,r}^{\eta,\varepsilon},\\
& w_{i,r}^{\eta} = g_{i,r}^{\eta,\varepsilon}-g_{i,r}^{\eta}, \bar w_{r}^{\eta}=\tfrac{1}{m}\textstyle\sum_{i=1}^m w_{i,r}^{\eta},   \nabla \tilde f_1^r =  \nabla \tilde f_1(\hat x_r,\xi_r).
\end{align*}
\end{definition}
Next, we present some preliminary results for the analysis.
\begin{lemma}[Preliminaries for Theorem\fy{s}~\ref{thm:main} and \ref{thm:asymp}]\em\label{lem:prelim}
Consider Algorithm~\ref{alg:FedStack} and Proposition \ref{prop:SGD}, where $\mathbb E[\|y_{\varepsilon_{i,r}}(\bullet)-y(\bullet)\|^2]=\mathcal O(1/T_{i,r})\leq \varepsilon_{i,r}$. Let Assumption \ref{assump:main} hold. Then\fy{,} the following hold for all $i\in[m]$ and $r\geq 0$.

\noindent (i) $\mathbb E[\bar g_{r}^{\eta} \mid \mathcal F_r] = \nabla f_2^{\eta}(\hat x_r)$.

\noindent (ii) $\mathbb E[\|\bar g_{r}^{\eta}\|^2] \leq 16 \sqrt{2\pi} (L_0^{\text{imp}})^2 n$.

\noindent (iii) $\mathbb E[\|w_{i,r}^{\eta}\|^2] \leq \tfrac{n^2}{\eta^2}(L_{0,y}^f)^2\varepsilon_{i,r}$.

\noindent (iv) $\mathbb E[\|\bar w_{r}^{\eta}\|^2] \leq \tfrac{n^2}{\eta^2}(L_{0,y}^f)^2\tfrac{1}{m}\textstyle\sum_{i=1}^m\varepsilon_{i,r}$.

\noindent (v) $\mathbb E[\bar g_{r}^{\eta}+\nabla \tilde f_1^r\mid \mathcal F_r]=\nabla f^{\eta}(\hat x_r)$.

\noindent (vi) $\mathbb E[\|\bar g_{r}^{\eta}+\nabla \tilde f_1^r\|^2] \leq 32\sqrt{2\pi} (L_0^{\text{imp}})^2 n+2\sigma^2+\mathbb E[\|\nabla f^{\eta}(\hat x_r)\|^2].$

\noindent (vii) $\mathbb E[\fy{\nabla} f^{\eta}(\hat x_r)^{\top}(\bar g_{r}^{\eta}+\nabla \tilde f_1^r)]=\mathbb E[\|\fy{\nabla} f^{\eta}(\hat x_r)\|^2]$.

\end{lemma}

\begin{proof}
\noindent (i) \fy{Invoking \yqr{Definition}~\ref{def:main}, we} have 
\begin{align*}
\mathbb E[\bar g_{r}^{\eta} \mid \mathcal F_r] &= \tfrac{1}{m}\textstyle\sum_{i=1}^m\mathbb E[g_{i,r}^{\eta} \mid \mathcal F_r]\\
&= \tfrac{1}{m}\textstyle\sum_{i=1}^m\mathbb E_{v_{i,r}}[\fy{\tfrac{n}{2\eta}}(f_{2,i,r}^{+}-f_{2,i,r}^{-})v_{i,r}]\\
&\stackrel{\text{Lemma }\ref{lem:SphericalSmooth} \text{ (i)}}{=}\nabla f_2^{\eta}(\hat x_r).
\end{align*}
\noindent (ii) We \fy{may write}
\begin{align*}
&\mathbb E[\|g_{i,r}^{\eta}\|^2 \mid \mathcal F_r]
= \mathbb E[\|\tfrac{n}{2\eta}(f_{2,i,r}^{+} - f_{2,i,r}^{-})v_{i,r}\|^2 \mid \mathcal F_r]\\
&=\tfrac{n^2}{4\eta^2}\mathbb E[(f_{2,i,r}^{+} - f_{2,i,r}^{-})^2 \mid \mathcal F_r]\\
&\leq \tfrac{n^2}{2\eta^2}\mathbb E[(f_{2,i,r}^{+} - \mathbb E_{\hat v}[f_2(\hat x_r+\eta \hat v,y_i(\hat x_r+\eta \hat v))])^2 \mid \mathcal F_r]\\
&+\tfrac{n^2}{2\eta^2}\mathbb E[(f_{2,i,r}^{-} - \mathbb E_{\hat v}[f_2(\hat x_r+\eta \hat v,y_i(\hat x_r+\eta \hat v))])^2 \mid \mathcal F_r],
\end{align*}
where $\hat v\in \mathbb S$. Then by the symmetric distribution of $v_{i,r}$ and $\hat v$, we have
\begin{align*}
&\mathbb E[(f_{2,i,r}^{+} - \mathbb E_{\hat v}[f_2(\hat x_r+\eta \hat v,y_i(\hat x_r+\eta \hat v))])^2 \mid \mathcal F_r]\\
&=\mathbb E[(f_{2,i,r}^{-} - \mathbb E_{\hat v}[f_2(\hat x_r+\eta \hat v,y_i(\hat x_r+\eta \hat v))])^2 \mid \mathcal F_r].
\end{align*}
From the previous two relations, we obtain
\begin{align*}
&\mathbb E[\|g_{i,r}^{\eta}\|^2 \mid \mathcal F_r]\leq \tfrac{n^2}{\eta^2}\mathbb E[(f_{2,i,r}^{+} - \mathbb E_{\hat v}[f_2(\hat x_r+\eta \hat v,y_i(\hat x_r+\eta \hat v))])^2 \mid \mathcal F_r].
\end{align*}
Next, define $h_i(v)=f_2(\hat x_r+\eta v,y_i(\hat x_r+\eta v))$, where $v\in\mathbb S$. Invoking Lemma \ref{lem:lip:hv}, we have $h$ is $\eta L_0^{\text{imp}}$-Lipschitz continuous. Then, we have
\begin{align*}
&\mathbb E[(f_{2,i,r}^{+} - \mathbb E_{\hat v}[f_2(\hat x_r+\eta \hat v,y_i(\hat x_r+\eta \hat v))])^2 \mid \mathcal F_r]\\
&=\int_0^{\infty}\mathbb P\left(\large|h_i(v_{i,r}) - \mathbb E_{\hat v}[h_i(\hat v)] \large|^2\geq \alpha \right) d\alpha \\
&= \int_0^{\infty}\mathbb P\left(\large|h_i(v_{i,r}) - \mathbb E_{\hat v}[h_i(\hat v)] \large|\geq \sqrt{\alpha} \right) d\alpha \\
&\stackrel{\text{Lemma }\ref{lem:Levy}}{\leq} \int_0^{\infty} 2\sqrt{2\pi} e^{-\tfrac{n\alpha}{8(\eta L_0^{\text{imp}})^2}}d\alpha = \tfrac{16 \sqrt{2\pi} (\eta L_0^{\text{imp}})^2}{n}.
\end{align*}
Combin\fy{ing the} preceding results, we obtain $\mathbb E[\|g_{i,r}^{\eta}\|^2 \mid \mathcal F_r] \leq 16 \sqrt{2\pi} (L_0^{\text{imp}})^2 n.$ Then, we have
\begin{align*}
\mathbb E[\|\bar g_{r}^{\eta}\|^2 \mid \mathcal F_r] &\leq \tfrac{1}{m}\sum_{i=1}^m\mathbb E[\|g_{i,r}^{\eta}\|^2 \mid \mathcal F_r]
=16 \sqrt{2\pi} (L_0^{\text{imp}})^2 n.
\end{align*}
Taking total expectation\yqr{s} on both sides, we obtain \fy{the result}.

\noindent (iii) We have 
\begin{align*}
&\mathbb E[\|w_{i,r}^{\eta}\|^2\mid \mathcal F_r]
=\mathbb E[\|g_{i,r}^{\eta,\varepsilon}-g_{i,r}^{\eta}\|^2\mid \mathcal F_r]\\
&=\tfrac{n^2}{4\eta^2} \mathbb E[\|f_{2,i,r}^{\varepsilon,+} - f_{2,i,r}^{\varepsilon,-}-f_{2,i,r}^{+} +f_{2,i,r}^{-}\|^2\mid \mathcal F_r]\\
&\leq \tfrac{n^2}{2\eta^2} \mathbb E[\|f_{2,i,r}^{\varepsilon,+} -f_{2,i,r}^{+} \|^2\mid \mathcal F_r]\\
&+\tfrac{n^2}{2\eta^2} \mathbb E[\| f_{2,i,r}^{\varepsilon,-}-f_{2,i,r}^{-}\|^2\mid \mathcal F_r]
\stackrel{\text{Assump. }\ref{assump:main}\text{ (i)}}{\leq} \tfrac{n^2}{\eta^2}(L_{0,y}^f)^2\varepsilon_{i,r}.
\end{align*}
Taking total expectation\yqr{s} on both sides, we obtain the bound in (iii).

\noindent (iv) We have 
$
\mathbb{E}\left[ \|\bar{w}^\eta_{r}\|^2 \right] \leq  \tfrac{1}{m}\textstyle\sum_{i=1}^m \mathbb{E}\left[\|{w_{i,r}^\eta}\|^2 \right]
\stackrel{\text{(iii)}}{\leq} \tfrac{n^2}{\eta^2} (L_{0,y}^{f})^2\tfrac{1}{m}\textstyle\sum_{i=1}^m\varepsilon_{i,r}.
$

\noindent (v) We have
$
\mathbb E[\bar g_{r}^{\eta}+\nabla \tilde f_1^r \mid \mathcal F_r] =\mathbb E[\bar g_{r}^{\eta} \mid \mathcal F_r]+\mathbb E[\nabla \tilde f_1^r \mid \mathcal F_r] \stackrel{\text{Lemma }\ref{lem:prelim}\text{ (i)}}{=}\nabla f^{\eta}(\hat x_r).
$
Taking total expectation\yqr{s} on both sides, we obtain \fy{the result}.

\noindent (vi) We have 
\begin{align*}
&\mathbb E[\|\bar g_{r}^{\eta}+\nabla \tilde f_1^r\|^2 \mid \mathcal F_r]\\
&=\mathbb E[\|\bar g_{r}^{\eta}+\nabla \tilde f_1^r-\nabla f^{\eta}(\hat x_r)+\nabla f^{\eta}(\hat x_r)\|^2 \mid \mathcal F_r]\\
&\stackrel{\text{(v)}}{=}\mathbb E[\|\bar g_{r}^{\eta}+\nabla \tilde f_1^r-\nabla f^{\eta}(\hat x_r)\|^2 +\|\nabla f^{\eta}(\hat x_r)\|^2\mid \mathcal F_r]\\
&=\mathbb E[\|\bar g_{r}^{\eta}+\nabla \tilde f_1^r-\nabla f_1(\hat x_r)-\nabla f_2^{\eta}(\hat x_r)\|^2 \mid\mathcal F_r]+\mathbb E[\|\nabla f^{\eta}(\hat x_r)\|^2 \mid \mathcal F_r]\\
&\leq 2\mathbb E[\|\bar g_{r}^{\eta}-\nabla f_2^{\eta}(\hat x_r)\|^2 \mid\mathcal F_r]\\
&+\mathbb E[\|\nabla f^{\eta}(\hat x_r)\|^2 \mid \mathcal F_r]+2\mathbb E[\|\nabla \tilde f_1^r-\nabla f_1(\hat x_r)\|^2 \mid\mathcal F_r]\\
&\stackrel{\text{(i), }\text{Assump. }\ref{assump:main}\text{ (iii)}}{\leq} 2\mathbb E[\|\bar g_{r}^{\eta}\|^2 \mid\mathcal F_r]+\mathbb E[\|\nabla f^{\eta}(\hat x_r)\|^2 \mid \mathcal F_r]+2\sigma^2\\
&=32\sqrt{2\pi} (L_0^{\text{imp}})^2 n+2\sigma^2+\mathbb E[\|\nabla f^{\eta}(\hat x_r)\|^2 \mid \mathcal F_r].
\end{align*}
Taking total expectation on both sides, we obtain the \fy{result}.

\noindent (vii) We have 
$
\mathbb E[\fy{\nabla} f^{\eta}(\hat x_r)^{\top}(\bar g_{r}^{\eta}+\nabla \tilde f_1^r) \mid \mathcal F_r]
=\fy{\nabla} f^{\eta}(\hat x_r)^{\top}\mathbb E[\bar g_{r}^{\eta}+\nabla \tilde f_1^r \mid \mathcal F_r] \stackrel{\text{(v)}}{=} \|\fy{\nabla} f^{\eta}(\hat x_r)\|^2.
$
Taking total expectation\yqr{s} on both sides, we obtain the bound.
\end{proof}

The next result will be employed in establishing convergence guarantees \yqr{for} Algorithm \ref{alg:FedStack}.
\begin{lemma}\em\label{lem:thms}
Consider Algorithm \ref{alg:FedStack}, and Definitions \ref{def:F:Alg} and \ref{def:main}. Let Assumption \ref{assump:main} hold. Let $\gamma_r \leq \tfrac{\eta}{4(\eta L_1^{f_1}+L_0^{\text{imp}}\sqrt{n})}$. \yq{Then,} the following \yq{holds}.
\begin{align*}
\mathbb E[f^\eta(\hat x_{r+1})\mid \mathcal F_r] &\leq f^\eta(\hat x_{r})-\tfrac{\gamma_r}{4} \|\nabla f^\eta(\hat x_{r})\|^2 + \gamma_r^2 c_1 + \gamma_r c_2\tfrac{1}{m}\textstyle\sum_{i=1}^m\varepsilon_{i,r},
\end{align*}
where $c_1=(32\sqrt{2\pi} L_0^{\text{imp}})^2 nL^f +2\sigma^2L^f$ and $c_2=\tfrac{3n^2}{4\eta^2}(L_{0,y}^f)^2$ are constants.

\end{lemma}
\begin{proof}
From Lemma \ref{lem:SphericalSmooth} (iv), we have $f_2^\eta(x,y_i(x))$ is $\tfrac{L_0^{\text{imp}}\sqrt{n}}{\eta}$-smooth in $x$. Therefore $f^\eta(x)=f_1(x)+f_2^\eta(x)$ is $\left(L_1^{f_1}+\tfrac{L_0^{\text{imp}}\sqrt{n}}{\eta}\right)$-smooth. Let $L^f\triangleq L_1^{f_1}+\tfrac{L_0^{\text{imp}}\sqrt{n}}{\eta}$\yq{. We} have  
\begin{align*}
f^\eta(\hat x_{r+1})&\leq f^\eta(\hat x_{r})+\nabla f^\eta(\hat x_{r})^{\top}(\hat x_{r+1}-\hat x_{r})+\tfrac{L^f}{2}\|\hat x_{r+1}-\hat x_{r}\|^2.
\end{align*}
From Algorithm \ref{alg:FedStack} and Definition \ref{def:main}, we have $\hat x_{r+1}  =  \hat x_r -\gamma_r( \nabla \tilde f_1^r +\tfrac{1}{m} \sum_{i=1}^m g_{i,r}^{\eta,\varepsilon})=\hat x_r -\gamma_r( \nabla \tilde f_1^r + \bar g_r^\eta +\bar w_r^\eta)$. Then
\begin{align*}
f^\eta(\hat x_{r+1})&\leq f^\eta(\hat x_{r})-\gamma_r \nabla f^\eta(\hat x_{r})^{\top}(\nabla \tilde f_1^r + \bar g_r^\eta +\bar w_r^\eta)+\tfrac{\gamma_r^2 L^f}{2}\| \nabla \tilde f_1^r + \bar g_r^\eta +\bar w_r^\eta\|^2\\
&\leq f^\eta(\hat x_{r})-\gamma_r \nabla f^\eta(\hat x_{r})^{\top}(\nabla \tilde f_1^r + \bar g_r^\eta)+\tfrac{\gamma_r}{2}(\|\nabla f^\eta(\hat x_{r})\|^2+\|\bar w_r^\eta\|^2) \\
&+ \gamma_r^2 L^f(\| \nabla \tilde f_1^r + \bar g_r^\eta\|^2+\|\bar w_r^\eta\|^2),
\end{align*}
where in the last inequality, we utilized the \yqr{identity} that $-a^{\top}b \leq \tfrac{1}{2}(\|a\|^2+\|b\|^2)$ and $\|a+b\|^2\leq 2\|a\|^2+2\|b\|^2$. Next, by taking conditional expectation\yqr{s} with respect to $\mathcal F_r$ on both sides and invoking Lemma \ref{lem:prelim} (vii), we obtain
\begin{align*}
\mathbb E[f^\eta(\hat x_{r+1})\mid \mathcal F_r] &\leq f^\eta(\hat x_{r})-\gamma_r \|\nabla f^\eta(\hat x_{r})\|^2+\tfrac{\gamma_r}{2}(\|\nabla f^\eta(\hat x_{r})\|^2+\mathbb E[\|\bar w_r^\eta\|^2\mid \mathcal F_r]) \\
&+ \gamma_r^2 L^f\mathbb E[\| \nabla \tilde f_1^r + \bar g_r^\eta\|^2+\|\bar w_r^\eta\|^2\mid \mathcal F_r].
\end{align*} 
Let $\gamma_r \leq \frac{1}{4L^f}$\yq{. By} invoking Lemma \ref{lem:prelim} (iv) and (vi), and combining terms, we obtain
\begin{align*}
\mathbb E[f^\eta(\hat x_{r+1})\mid \mathcal F_r] &\leq f^\eta(\hat x_{r})-\tfrac{\gamma_r}{4} \|\nabla f^\eta(\hat x_{r})\|^2 + \gamma_r^2(32\sqrt{2\pi} (L_0^{\text{imp}})^2 n L^f+2\sigma^2L^f)\\
&+\tfrac{3\gamma_r n^2}{4\eta^2}(L_{0,y}^f)^2\tfrac{1}{m}\textstyle\sum_{i=1}^m\varepsilon_{i,r}.
\end{align*}
\end{proof}
We now \fy{present} nonasymptotic guarantees for Alg.~\ref{alg:FedStack}.
\begin{theorem}[Rate and complexity statement for problem \eqref{prob::main}]\em\label{thm:main}
Consider Algorithms \ref{alg:FedStack} and \ref{alg:SGD} for solving problem \eqref{prob::main}. \yq{Let} Assumption \ref{assump:main} hold. Let $\gamma_r\leq \tfrac{\eta}{4(\eta L_1^{f_1}+L_0^{\text{imp}}\sqrt{n})}$, and $r^*$ be uniformly sampled from $0,\dots,R-1$.

\noindent (i) Consider Algorithm \ref{alg:SGD}. \yqr{L}et $\tilde\gamma_{i,t}\triangleq\tfrac{\tilde\gamma_i}{t+\Gamma}\leq \tfrac{\mu_{h_i}}{2(L_{1,y}^{h_i})^2}$, where $\Gamma, \tilde\gamma_i>0$.  Then
\begin{align*}
\mathbb E[\|y_{\varepsilon_{i,r}}(\bullet)-y(\bullet)\|^2]\leq \varepsilon_{i,r} = \mathcal O(1/H_{i,r}).
\end{align*}

\noindent (ii) {\bf [Error bound]} Let $\gamma_r=\tfrac{c}{\sqrt{r+1}}$, where $c=\frac{1}{4L^f}$, $L^f\triangleq L_1^{f_1}+\tfrac{L_0^{\text{imp}}\sqrt{n}}{\eta}$. Let $H_{i,r}=\lceil\tau_i \sqrt{r+1}\rceil$, where $\tau_i\geq 1$. We have
\begin{align*}
\mathbb E[\|\nabla f^\eta(\hat x_{r^*})\|^2] &\leq \tfrac{\Theta_1+\Theta_2+\Theta_3}{\sqrt R},
\end{align*}
where $\Theta_1=4L^f(|\mathbb E[f^\eta(\hat x_{0})]|+|f^{\eta}_{*}|)$, $\Theta_2=64\sqrt{2\pi} (L_0^{\text{imp}})^2 n+4\sigma^2$, and $\Theta_3=\tfrac{6 n^2}{\eta^2}(L_{0,y}^f)^2$.

\noindent (iii) {\bf [Communication complexity]}
Let $\epsilon>0$ be an arbitrary scalar and $R$ denote the number of communications such that $\mathbb E[\|\nabla f^\eta(\hat x_{r^*})\|^2]\leq \epsilon$. Then the communication complexity is
$$R=\mathcal O \left(\left(\tfrac{L_0^{\text{imp}} \sqrt n}{\eta}+(L_0^{\text{imp}})^2 n+ \tfrac{ n^2}{\eta^2}\right)^2\epsilon^{-2}\right).$$

\noindent (iv) {\bf [Iteration complexity of client $i$]} Let $K_i\triangleq \sum_{r=0}^{R-1} H_{i,r} $\yq{. Then,} we have
$$K_i=\mathcal O \left(\left(\tfrac{L_0^{\text{imp}} \sqrt n}{\eta}+(L_0^{\text{imp}})^2 n+\tfrac{ n^2}{\eta^2}\right)^3\epsilon^{-3}\right).$$

\end{theorem}

\begin{proof}
\noindent (i) See Proposition \ref{prop:SGD}.
 
\noindent (ii) Let $\gamma_r \leq \frac{1}{4L^f}$, where $L^f\triangleq L_1^{f_1}+\tfrac{L_0^{\text{imp}}\sqrt{n}}{\eta}$. By invoking Lemma \ref{lem:thms}, and taking expectations on both sides, we obtain
\begin{align*}
\mathbb E[f^\eta(\hat x_{r+1})] &\leq \mathbb E[f^\eta(\hat x_{r})]-\tfrac{\gamma_r}{4} \mathbb E[\|\nabla f^\eta(\hat x_{r})\|^2] + \gamma_r^2(32\sqrt{2\pi} (L_0^{\text{imp}})^2 n L^f+2\sigma^2L^f)\\
&+\tfrac{3\gamma_r n^2}{4\eta^2}(L_{0,y}^f)^2\tfrac{1}{m}\textstyle\sum_{i=1}^m\varepsilon_{i,r}.
\end{align*}
Let $\gamma_r=\tfrac{(4L^f)^{-1}}{\sqrt{r+1}}$. Divide both side by $\tfrac{\gamma_r R}{4}$ and summing the relation on $r$ from $0$ to $R-1$\yq{. W}e obtain
\begin{align*}
&\tfrac{1}{R}\sum_{r=0}^{R-1}\mathbb E[\|\nabla f^\eta(\hat x_{r})\|^2] \leq \tfrac{4L^f(|\mathbb E[f^\eta(\hat x_{0})]|+|f^{\eta}_{*}|)}{\sqrt R}  + \tfrac{64\sqrt{2\pi} (L_0^{\text{imp}})^2 n +4\sigma^2}{\sqrt R}
+\tfrac{3 n^2}{\eta^2}(L_{0,y}^f)^2\tfrac{1}{m}\textstyle\sum_{i=1}^m\frac{1}{R}\textstyle\sum_{r=0}^{R-1}\varepsilon_{i,r},
\end{align*}
where we utilized $\mathbb E[f^\eta(\hat x_{R})]\geq f^{\eta}_{*}$, $a-b\leq |a|+|b|$, and $\sum_{r=0}^{R-1} \tfrac{1}{\sqrt{r+1}}\leq 2\sqrt R$ in the preceding relation, where $f^{\eta}_{*}$ denotes the infimum of $f^\eta$.

Let $H_{i,r}=\lceil\tau_i \sqrt{r+1}\rceil$, where $\tau_i\geq 1$. In view of $\varepsilon_{i,r}=\mathcal O(1/H_{i,r})$, we have
\begin{align*}
\tfrac{1}{m}\textstyle\sum_{i=1}^m\tfrac{1}{R}\textstyle\sum_{r=0}^{R-1}\varepsilon_{i,r} &= \tfrac{1}{m}\textstyle\sum_{i=1}^m\tfrac{1}{\tau_i}\tfrac{1}{R}\textstyle\sum_{r=0}^{R-1}\tfrac{1}{\sqrt{r+1}}\\
&\leq \tfrac{1}{m}\textstyle\sum_{i=1}^m\tfrac{1}{\tau_i} \tfrac{2}{\sqrt R} \leq \tfrac{2}{\sqrt R}.
\end{align*}
Therefore, from preceding relations, we obtain
\begin{align*}
\tfrac{1}{R}\textstyle\sum_{r=0}^{R-1}\mathbb E[\|\nabla f^\eta(\hat x_{r})\|^2] &\leq \tfrac{4L^f(|\mathbb E[f^\eta(\hat x_{0})]|+|f^{\eta}_{*}|)}{\sqrt R} + \tfrac{64\sqrt{2\pi} (L_0^{\text{imp}})^2 n+4\sigma^2}{\sqrt R}+\tfrac{\tfrac{6 n^2}{\eta^2}(L_{0,y}^f)^2}{\sqrt R}.
\end{align*}

\noindent (iii) Let $\mathbb E[\|\nabla f^\eta(\hat x_{r^*})\|^2]\leq \epsilon$\yq{. Then,} by the relation in (ii), we obtain the desired result. 

\noindent (iv)
Let $K_i$ be number of total iterations by client $i$. We have
$
K_i =\sum_{r=0}^{R-1} \tau_i \sqrt{r+1} \leq \tau_i\int_1^{R} \sqrt r \leq \tfrac{2\tau_i}{3} R^{3/2}.
$
Then, by the relation in (iii), we obtain the result.
\end{proof}

The next result will be employed in establishing  asymptotic convergence guarantee\fy{s for} Algorithm \ref{alg:FedStack}.
\begin{lemma}\em\label{lem:gamma:thm2}
Let $\{\gamma_r\}$ be a nonnegative, nonsummable and square-summable sequence. Further, let sequence $\{\gamma_r\tfrac{1}{m}\textstyle\sum_{i=1}^m\varepsilon_{i,r}\}$ to be nonnegative and summable. Then we have
$\sum_{r=0}^{\infty} \left(\gamma_r^2 c_1 + \gamma_r c_2\tfrac{1}{m}\textstyle\sum_{i=1}^m\varepsilon_{i,r}\right) <\infty ,$
where $c_1$ and $c_2$ are positive constants.
\end{lemma}

\begin{remark}
We note that the requirements on sequences $\{\gamma_r\}$ and $\{\gamma_r\tfrac{1}{m}\textstyle\sum_{i=1}^m\varepsilon_{i,r}\}$ in Lemma \ref{lem:gamma:thm2} are indeed realistic. A simple example would be to let $\gamma_r=\tfrac{c}{(r+1)^p}$, where $p\in(\tfrac{1}{2},1]$, $c>0$; then set $\varepsilon_{i,r}=\mathcal O(\tfrac{1}{(r+1)^q})$, where $q\geq\tfrac{1}{2}$.
\end{remark}

Next, we formally present asymptotic convergence guarantee for our methods in solving \fy{problem} \eqref{prob::main}.
\begin{theorem}[Asymptotic guarantee for Algorithm \ref{alg:FedStack}]\em\label{thm:asymp}
Consider Algorithms \ref{alg:FedStack} and \ref{alg:SGD} for solving \fy{p}roblem \eqref{prob::main} \fy{and} let Assumption \ref{assump:main} hold. Let $\gamma_r\leq \tfrac{\eta}{4(\eta L_1^{f_1}+L_0^{\text{imp}}\sqrt{n})}$. Let $\{\gamma_r\}$ be a nonnegative, nonsummable and square-summable sequence. Further, let \fy{the} sequence $\{\gamma_r\tfrac{1}{m}\textstyle\sum_{i=1}^m\varepsilon_{i,r}\}$ to be nonnegative and summable. Then we have
\begin{align*}
\yq{\liminf_{r\to \infty}\|\nabla f^{\eta} (\hat x_r)\| = 0 \,\, \text{ almost surely.}}
\end{align*}
\begin{proof}
From Lemma \ref{lem:thms} and \ref{lem:gamma:thm2}, we have
\begin{align*}
\mathbb E[f^\eta(\hat x_{r+1})-f^{\eta}_*\mid \mathcal F_r] &\leq f^\eta(\hat x_{r})-f^{\eta}_*-\tfrac{\gamma_r}{4} \|\nabla f^\eta(\hat x_{r})\|^2 + \gamma_r^2 c_1 + \gamma_r c_2\tfrac{1}{m}\textstyle\sum_{i=1}^m\varepsilon_{i,r},
\end{align*}
and
$
\sum_{r=0}^{\infty} \left(\gamma_r^2 c_1 + \gamma_r c_2\tfrac{1}{m}\textstyle\sum_{i=1}^m\varepsilon_{i,r}\right) <\infty ,
$
where $c_1=(32\sqrt{2\pi} L_0^{\text{imp}})^2 nL^f +2\sigma^2L^f$, $c_2=\tfrac{3n^2}{4\eta^2}(L_{0,y}^f)^2$ and $f^{\eta}_*\triangleq \inf_{\fy{x}} f^{\eta}(x)$. Let $X_r=\tfrac{\gamma_r}{4} \|\nabla f^\eta(\hat x_{r})\|^2$, $Y_r=f^\eta(\hat x_{r})-f^{\eta}_*$, and $Z_r=\gamma_r^2 c_1 + \gamma_r c_2\tfrac{1}{m}\textstyle\sum_{i=1}^m\varepsilon_{i,r}$. \yqr{Invoking} Lemma \ref{lem:robsieg}, \yqr{we} obtain
$
\sum_{r=0}^{\infty} \gamma_r \|\nabla f^\eta(\hat x_{r})\|^2 <\infty
$ \fy{almost surely.}
\yq{Since $\sum_{r=0}^{\infty} \gamma_r =\infty$, we have $\displaystyle\liminf_{r\to \infty}\|\nabla f^{\eta} (\hat x_r)\| = 0$}. 
\end{proof}
\end{theorem}

\begin{remark}
Theorem \ref{thm:main} is equipped with guarantees \yqr{for} convergence rate, communication and iteration complexity. In (ii), we set $H_{i,r}=\mathcal O(\sqrt{r+1})$ in the analysis, but in practice, \yqr{one may} choose $H_{i,r}$ such that the sum of all the \fy{inexact errors, i.e.,} $\sum_{r=0}^{R-1} \varepsilon_{i,r}\leq \mathcal O(\sqrt{R})$, \fy{leading to the convergence rate of} $\mathcal O(1/\sqrt{R})$. We note that the rate is given in terms of communication rounds instead of iteration number. This is indeed because we \yqr{employed} an implicit programming approach, and the zeroth-\fy{order} gradient steps on the implicit function only \yqr{occur} at \fy{global steps}. Theorem \ref{thm:asymp} provides \fy{an} asymptotic guarantee, which appears to be \yqr{among the first} in FL when the aggregation function is nonconvex and nonsmooth, \fy{extending the results in centralized settings in \cite{marrinan2023zeroth}}. 
\end{remark}

\section{Numerical Experiments} 
We consider image classification tasks with MNIST, CIFAR-10, and Fashion-MNIST dataset. Throughout the experiment\fy{s}, we simulate the non-iid setting with Dirichlet distribution $Dir(\alpha)$, which is commonly used in FL experiments (e.g. \cite{li2022federated}). $\alpha$ is the concentration parameter that is used to determine the non-iid level. We consider $\alpha=0.1$ for non-iid settings and $\alpha\to \infty$ for iid settings. \yq{We note that when $\alpha$ is small (e.g., $\alpha\leq 1$), all clients only possess a small subset of all classes, making the data distribution among clients nonidentical.} Fig. \ref{fig:vis:noniid} \yqr{provides a visual illustration of} the non-iid data distribution in the experiments. 
\begin{figure}[htbp]
  \centering
  \includegraphics[width=0.45\textwidth]{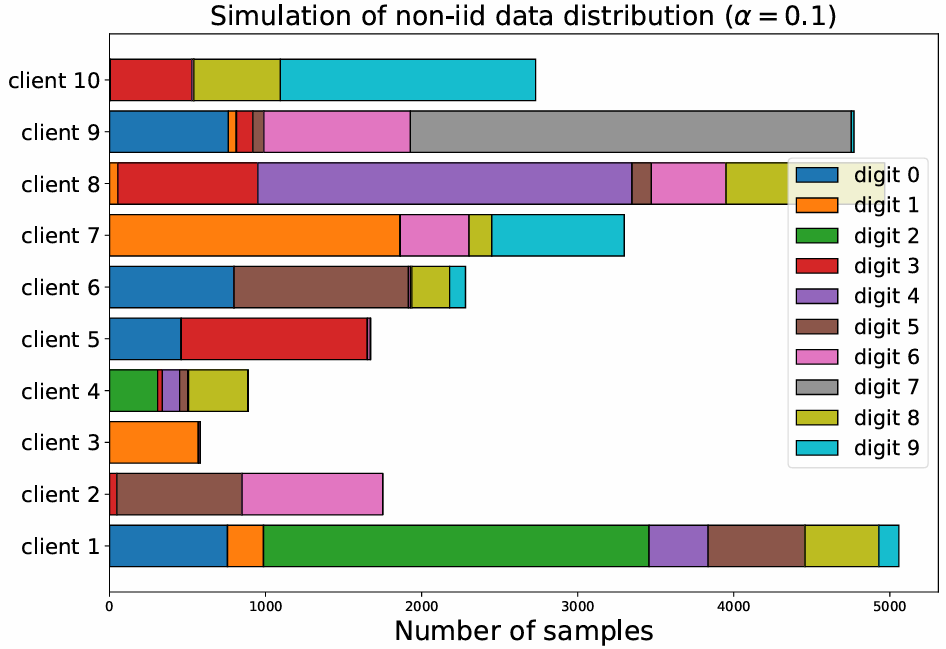}
  \caption{Non-iid data simulation across clients with MNIST.}
  \label{fig:vis:noniid}
\end{figure}

We consider cross-entropy loss, where $f_1(x)=-\tfrac{1}{N_s}\sum_{j=1}^{N_s} \sum_{c=1}^C I_{jc} \log (p_{jc}^x)$, $f_2(x)=\tfrac{\lambda}{2}\sum_{i=1}^m \rho_i \|x-y_i(x)\|^2$, $y_i(x) = \text{arg}\min_{y_i}  -\tfrac{1}{N_i}\sum_{j=1}^{N_i} \sum_{c=1}^C I_{jc} \log (p_{jc}^{y_i}) + \tfrac{\mu}{2}\|x-y_i\|^2$. Let $N_s$ \yqr{and} $N_i$ denote the number of data samples assigned to the server-agent and client-agent $i$, respectfully, $C$ is the number of classes, $I_{jc}=1$ if sample $u_j$ belongs to class $c$, else $I_{jc}=0$. $p_{jc}^x=e^{u_j^Tx_c}/\sum_{h=1}^Ce^{u_j^Tx_h}$, $p_{jc}^{y_i}=e^{u_j^Ty_c^i}/\sum_{h=1}^Ce^{u_j^Ty_h^i}$, and $\rho_i=N_i/N_{tr}$. $\lambda$ and $\mu$ are positive scalars, $x=[x_c]_{c=1}^C\in \mathbb R^{n\times C}$ and $y_i=[y_c^i]_{c=1}^C\in \mathbb R^{n\times C}$.

Throughout the experiments, we set $\gamma_r=\tfrac{0.01}{\sqrt{r+1}}$, $\tilde \gamma_{i,t}=\tfrac{0.1}{t+1}$, and $\eta=0.1$. The data is distributed as follows: first we split the data into $90\%$ training and $10\%$ testing data. \yqr{T}hen we assign $30\%$ of the training data to the server-agent and the rest to client-agent\yqr{s} with Dirichlet distribution. We simulate the straggler effect by sampling a subset $S_r \subset[m]$ during communication round $r$; we note that in our scheme, \yqr{setting} $H_{i,r}=0$ for some $r$ have a similar effect. We consider $m=10$ client-agents in the experiments.

\noindent {\bf (i) Convergence behavior.} We show the convergence behavior of Algorithm \ref{alg:FedStack} in terms of \yqr{the} global loss in Fig.~\ref{fig:Noniid_loss}. We denote by $\beta$ the proportion of participating clients in each communication round. We set the number of local steps \yqr{$H_{i,r} = \lceil\tau \sqrt{r+1}\rceil$} to match with \fy{our} theory, and choose $\tau \in \{5,20,50\}$. We consider three settings with \fy{an} increasing level of heterogeneity (in terms of both data and system), each setting is assigned with a pair of $(\alpha,\beta)$: $(1000,90\%)$ represents a homogeneous setting; $(1,50\%)$ represents a moderately heterogeneous setting; and $(0.1,10\%)$ represents an extreme heterogeneous setting. We observe that the global loss converge\yqr{s} faster with a larger number of local steps under homogeneous and moderately heterogeneous environments. However, under extreme heterogeneity, local steps \yqr{are} not contributing to faster convergence. This is \yqr{because} only \fy{one} client communicates per round, more local steps \fy{may lead} to a biased solution towards that client. This can be handled by carefully \fy{increasing} the value of $\mu$, so that the local models remain closer to the global model.

\begin{figure*}[ht]
\centering
\begin{tikzpicture}
\node at (0, 0) [rotate=90] {\textbf{Fashion-MNIST}};
\node[anchor=west] (fig1) at (0.5, 0) {\includegraphics[width=0.7\textwidth]{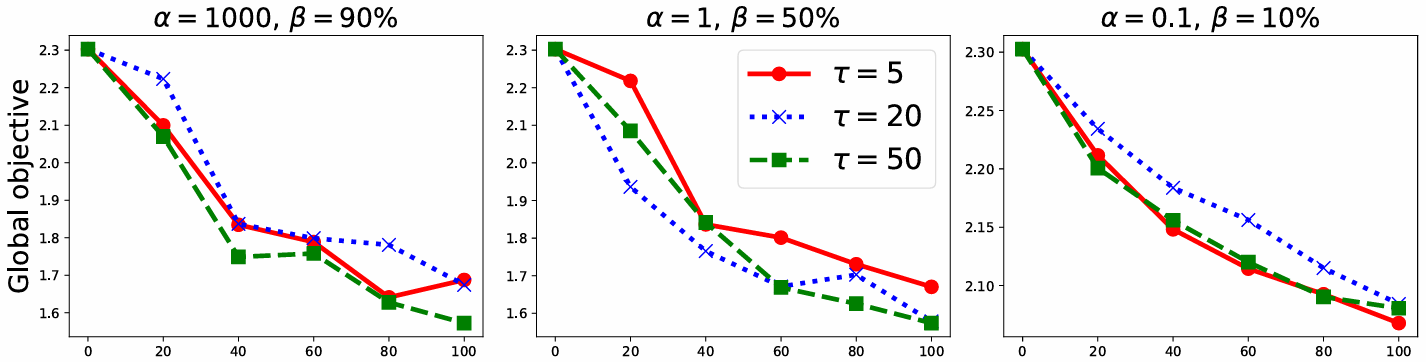}};

\node at (0, -3) [rotate=90] {\textbf{CIFAR-10}};
\node[anchor=west] (fig2) at (0.5, -3) {\includegraphics[width=0.7\textwidth]{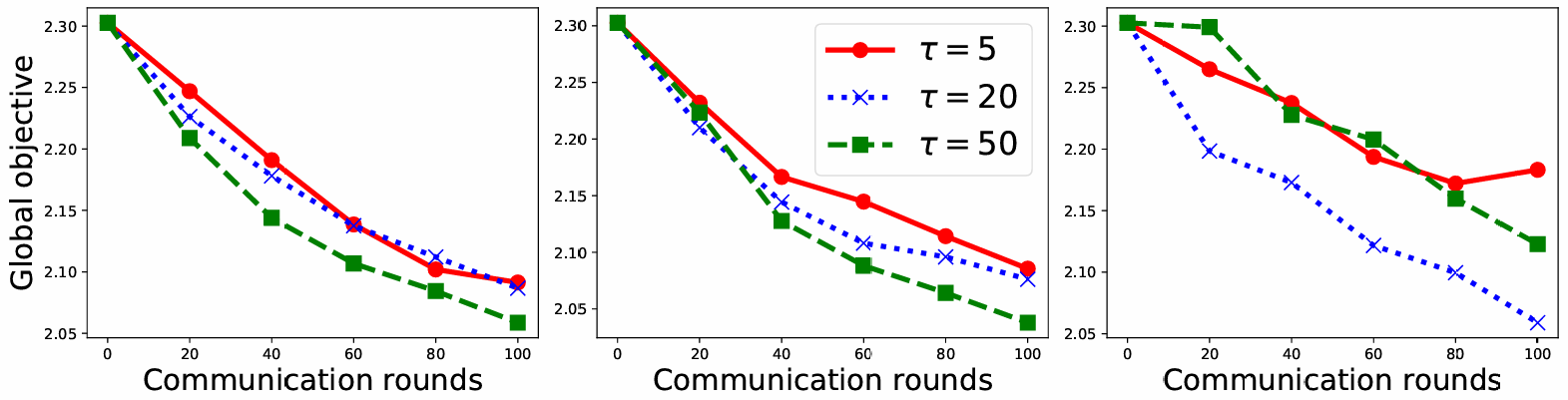}};
\end{tikzpicture}
\caption{Comparison of different $\tau$ under increasing heterogeneity with Fashion-MNIST and CIFAR-10.}
\label{fig:Noniid_loss}
\end{figure*}

\noindent {\bf (ii) Testing accuracy.} We compare the testing accuracy of our method in Algorithm \ref{alg:FedStack} with \yqr{that of} FedAvg \cite{mcmahan2017communication}, SCAFFOLD \cite{karimireddy2020scaffold}, and FedProx \cite{li2020federated}. We set $\tau=20$, and for fair comparison, we set the total number of local steps to be the same across the four methods. We set $R=500$ in this experiment. Similar to the previous experiment, we consider \yqr{three} combinations of $(\alpha,\beta)$. The test accuracy \yqr{results are presented} in Table~\ref{tbl:test:acc2}. We observe that {\bf ZO-HFL} demonstrates robust performance under various settings and \yqr{achieves} the best test accuracy in five out \fy{of} six heterogeneous environments, and stays competitive in the homogeneous setting.


\begin{table*}[ht]
\caption{Test accuracy on the {\it global model} with increasing heterogeneity level}
\centering
{\tiny
\begin{tabular}{|c|ccc|ccc|ccc|}
\hline
\multirow{2}{*}{} & \multicolumn{3}{c|}{$\alpha=1000,$ $\beta = 90\%$}                                & \multicolumn{3}{c|}{$\alpha=1,$ $\beta=50\%$}                                             & \multicolumn{3}{c|}{$\alpha=0.1,$ $\beta=10\%$}                                                 \\ \cline{2-10} 
                  & \multicolumn{1}{c|}{MNIST}     & \multicolumn{1}{c|}{Fa.-MNIST} & CIFAR-10  & \multicolumn{1}{c|}{MNIST}     & \multicolumn{1}{c|}{Fa.-MNIST}   & CIFAR-10        & \multicolumn{1}{c|}{MNIST}           & \multicolumn{1}{c|}{Fa.-MNIST}   & CIFAR-10        \\ \hline
FedAvg \cite{mcmahan2017communication}    & \multicolumn{1}{c|}{$87.17\%$} & \multicolumn{1}{c|}{$77.52\%$}     & $68.36\%$ & \multicolumn{1}{c|}{$77.33\%$} & \multicolumn{1}{c|}{$59.64\%$}       & $37.20\%$       & \multicolumn{1}{c|}{$39.19\%$}       & \multicolumn{1}{c|}{$45.50\%$}       & $27.43\%$       \\ \hline
FedProx \cite{li2020federated}   & \multicolumn{1}{c|}{$86.90\%$} & \multicolumn{1}{c|}{$77.34\%$}     & $67.52\%$ & \multicolumn{1}{c|}{$75.61\%$} & \multicolumn{1}{c|}{$60.28\%$}       & $38.55\%$       & \multicolumn{1}{c|}{$45.15\%$}       & \multicolumn{1}{c|}{$49.44\%$}       & $29.03\%$       \\ \hline
SCAFFOLD \cite{karimireddy2020scaffold}  & \multicolumn{1}{c|}{$91.54\%$} & \multicolumn{1}{c|}{$82.25\%$}     & $72.36\%$ & \multicolumn{1}{c|}{$91.25\%$} & \multicolumn{1}{c|}{$83.48\%$}       & $50.41\%$       & \multicolumn{1}{c|}{$87.36\%$}       & \multicolumn{1}{c|}{$74.91\%$}       & $44.65\%$       \\ \hline
{\bf ZO-HFL}      & \multicolumn{1}{c|}{$90.82\%$} & \multicolumn{1}{c|}{$78.51\%$}     & $68.12\%$ & \multicolumn{1}{c|}{$88.44\%$} & \multicolumn{1}{c|}{${\bf 85.51\%}$} &  ${\bf 58.06\%}$ & \multicolumn{1}{c|}{${\bf 87.70\%}$} & \multicolumn{1}{c|}{ ${\bf 76.86\%}$} & ${\bf 52.51\%}$ \\ \hline
\end{tabular}}
\label{tbl:test:acc2}
\end{table*}

\section{Concluding Remarks}
We introduce a novel hierarchical optimization framework explicitly designed to address heterogeneity in federated learning. We design and analyze a zeroth-order implicit federated algorithm, {\bf ZO-HFL}, equipped with both asymptotic and nonasymptotic convergence, iteration and communication complexity guarantees. Our approach enhances traditional FL methodologies by enabling both global and personalized training. This structure not only maintains a global model at the server level, but also facilitates personalized adaptations at the client level. It also allows for nonidentical number of local steps among clients at each communication round, providing significant flexibility compared with standard FL methods for resolving system heterogeneity. 

\section{Appendix}
\begin{lemma}[Randomized spherical smoothing]\label{lem:SphericalSmooth}\em      
Let $h:\mathbb{R}^n\to \mathbb{R}$ be a given continuous function and define $h^{\eta}(x)\triangleq \mathbb{E}_{u \in \mathbb{B}} \left[ h(x+ \eta u)  \right].$ Then, the following hold. 

\noindent (i) $h^\eta$ is continuously differentiable and $
\nabla h^\eta(x) = \tfrac{n}{\eta} \mathbb{E}_{v \in \mathbb{S}}  \left[  h(x+ \eta v)v \right]$ for any $x\in \mathbb{R}^n$, and $
\nabla h^\eta(x) =\tfrac{n}{2\eta}  \mathbb{E}_{v \in \mathbb{S}}  \left[  (h(x+\eta v)-  h(x-\eta v)) v  \right]$.

Suppose $h$ is Lipschitz continuous with parameter $L_0>0$. Then, the following statements hold.

\noindent (ii) $| h^{\eta}(x) -  h^{\eta}(y)| \leq L_0 \|x-y \|$ for all $x,y\in \mathbb{R}^n$; 

\noindent (iii) $| h^{\eta}(x) -  h(x)| \leq L_0\eta$ for all $x\in \mathbb{R}^n$; 

\noindent (iv) $\| \nabla h^{\eta}(x) - \nabla h^{\eta} (y)\| \leq \frac{L_0 \sqrt{n}}{\eta}  \|x-y \|$ for all $x,y\in \mathbb{R}^n$, $c>0$ is a constant.
\end{lemma}
\begin{proof}
(i) From Lemma 1 in \cite{cui2023complexity}, we have $\nabla h^\eta(x) = \tfrac{n}{\eta}\mathbb{E}_{v \in \mathbb{ S}}  \left[  h(x+\eta v)  v  \right]$ for any $x\in \mathbb{R}^n$. By the symmetric property of the distribution of $v$, we have $\mathbb{E}_{v \in \mathbb{ S}}  \left[  h(x+\eta v) v  \right]=\mathbb{E}_{v \in \mathbb{ S}}  \left[  h(x-\eta v) (-v) \right]$. Therefore, we have $\tfrac{n}{2\eta}\mathbb{E}_{v}  \left[(h(x+ \eta v)-h(x-\eta v)) v \right]=\tfrac{n}{2\eta}\mathbb{E}_{v}  \left[  h(x+\eta v) v \right]+\tfrac{n}{2\eta}\mathbb{E}_{v}  \left[h(x-\eta v)(-v)  \right]=\tfrac{n}{\eta}\mathbb{E}_{v}  \left[  h(x+ v)v  \right] = \nabla h^\eta(x)$.

(ii, iii) See Lemma 1 in \cite{cui2023complexity}.

(iv) From  \cite[Lemma 8]{YNS_Automatica12}, we have $\| \nabla h^{\eta}(x) - \nabla h^{\eta} (y)\| \leq \frac{C_nL_0}{\eta}  \|x-y \|$, where $C_n=\tfrac{2}{\pi}\tfrac{n!!}{(n-1)!!}$ if $n$ is even, and $C_n=\tfrac{n!!}{(n-1)!!}$ if $n$ is odd. It remains to show that $C_n\leq \sqrt n$.

Let $w_n=\int_0^{\pi/2}\text{sin}^n(x)dx$ denotes the Wallis's integral, we have two properties: (i) $w_n=\tfrac{n-1}{n}w_{n-2}$ for $n\geq 2$, where we define $w_0=\tfrac{\pi}{2}$ and $w_1=1$; (ii) $w_{n+1}<w_n$ \cite{spivak2006calculus}. Then, by simple induction, we have $w_n = \tfrac{\pi}{2}\tfrac{(n-1)!!}{n!!}$ for all even $n$, and $w_n = \tfrac{(n-1)!!}{n!!}$ for all odd $n$. Therefore, we have $C_n=\tfrac{1}{w_n}$.

Next, by noting that $(n+1)w_nw_{n+1}=(n+1)\tfrac{\pi}{2}\tfrac{(n-1)!!}{n!!}\tfrac{n!!}{(n+1)!!}=\tfrac{\pi}{2}$, and by property (ii), we have $w_n > \sqrt{\tfrac{\pi}{2(n+1)}} = \sqrt{\tfrac{\pi n}{2(n+1)}}\tfrac{1}{\sqrt n} \geq \sqrt{\tfrac{\pi }{3}}\tfrac{1}{\sqrt n}$, for $n\geq 2$.

Therefore, we obtain $C_n=\tfrac{1}{w_n}<\sqrt{\tfrac{3}{\pi}}\sqrt n < \sqrt n$.
\end{proof}

\noindent \fy{\bf An example of nonsmooth nonconvex implicit function.} Consider a bilevel optimization problem given as
\begin{align}
\begin{aligned}
&\min_{x\in \mathbb R^n} \, \tfrac{1}{2}\|x+{\bf 1}_n -y(x)\|^2, \\
&\,\, \text{s.t. }y(x)=\text{arg min}_{y\in \mathbb R_{+}^n} \|y-x\|^2.
\end{aligned}
\end{align}
Let $f(\bullet)=\tfrac{1}{2}\|\bullet+{\bf 1}_n -y(\bullet)\|^2$ denotes the implicit function, where ${\bf 1}_n$ denotes an n-dimensional vector with all elements equal to $1$.


Fig. \ref{fig:NDNCBiLv} shows function $f(x)$. Notably, it is nondifferentiable at points where at least one element $x_i=0$.
\begin{figure}[htbp]
  \centering
  \includegraphics[width=0.45\textwidth]{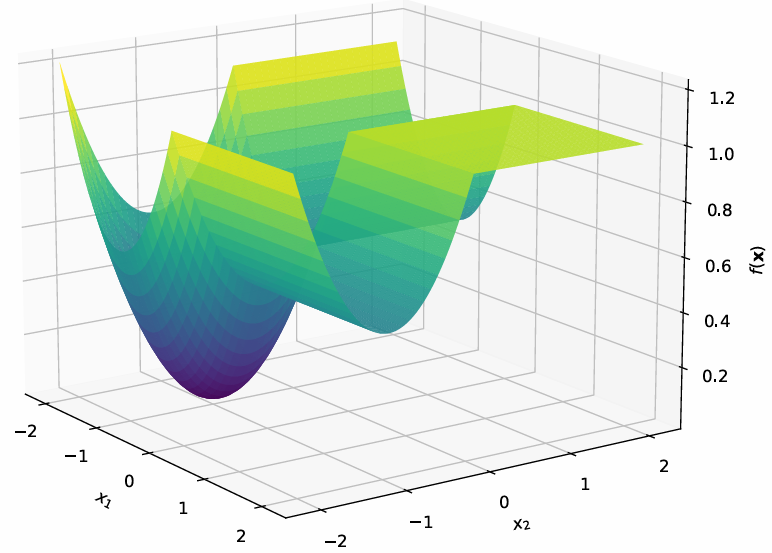}
  \caption{The implicit function $f(x)$ when $n=2$.}
  \label{fig:NDNCBiLv}
\end{figure} 

\bibliographystyle{siam}
\bibliography{ref}

\end{document}